\documentclass[12pt]{amsart}
\usepackage[colorlinks=true,urlcolor=blue, citecolor=red,linkcolor=blue,linktocpage,pdfpagelabels, bookmarksnumbered,bookmarksopen]{hyperref}
\usepackage[hyperpageref]{backref}
\usepackage{amsthm} 
\usepackage{latexsym,amsmath,amssymb}
\usepackage{accents}
\usepackage[colorinlistoftodos,prependcaption,textsize=tiny]{todonotes}
\usepackage{a4wide}
\usepackage{soul}
\usepackage{mathtools} 
\usepackage{xparse} 
  
\title{A remark on $C^{1,\alpha}$-regularity for differential inequalities in viscosity sense}

\author{Armin Schikorra}
\address[Armin Schikorra]{Department of Mathematics,
University of Pittsburgh,
301 Thackeray Hall,
Pittsburgh, PA 15260, USA}
\email{armin@pitt.edu}

plus 6pt minus 12pt
\def\eps{\varepsilon}

\def\N{{\mathbb N}}

\newtheorem{theorem}{Theorem}
\newtheorem{lemma}[theorem]{Lemma}
\newtheorem{corollary}[theorem]{Corollary}
\newtheorem{proposition}[theorem]{Proposition}

\def\osc{\mathop{\rm osc\,}}

\def\dist{{\rm dist\,}}

\newcommand{\R}{\mathbb{R}}

\newcommand{\brac}[1]{\left (#1 \right )}

\newcommand{\barint}{
\rule[.036in]{.12in}{.009in}\kern-.16in \displaystyle\int }

\newcommand{\barcal}{\mbox{$ \rule[.036in]{.11in}{.007in}\kern-.128in\int $}}

\def\mvint_#1{\mathchoice
          {\mathop{\vrule width 6pt height 3 pt depth -2.5pt
                  \kern -8pt \intop}\nolimits_{\kern -3pt #1}}%
          {\mathop{\vrule width 5pt height 3 pt depth -2.6pt
                  \kern -6pt \intop}\nolimits_{#1}}%
          {\mathop{\vrule width 5pt height 3 pt depth -2.6pt
                  \kern -6pt \intop}\nolimits_{#1}}%
          {\mathop{\vrule width 5pt height 3 pt depth -2.6pt
                  \kern -6pt \intop}\nolimits_{#1}}}

\numberwithin{theorem}{section} \numberwithin{equation}{section}

\newcommand{\lap}{\Delta }
\newcommand{\aleq}{\precsim}

\renewcommand{\div}{\operatorname{div}}

\let\latexchi\chi
\makeatletter
\renewcommand\chi{\@ifnextchar_\sub@chi\latexchi}
\newcommand{\sub@chi}[2]{
  \@ifnextchar^{\subsup@chi{#2}}{\latexchi^{}_{#2}}%
}
\newcommand{\subsup@chi}[3]{
  \latexchi_{#1}^{#3}%
}
\makeatother

\begin{document}
\begin{abstract}
We prove interior $C^{1,\alpha}$-regularity for solutions
\[
 - \Lambda \leq F(D^2 u) \leq \Lambda
\]
where $\Lambda$ is a constant and $F$ is fully nonlinear, 1-homogeneous, uniformly elliptic. 

The proof is based on a reduction to the homogeneous equation $F(D^2u) = 0$ by a blow-up argument -- i.e. just like what is done in the case of viscosity solutions $F(D^2 u) = f$ for $f \in L^\infty$. 

However it was not clear to us that the above inequality implies $F(D^2 u) = f$ for some bounded $f$ (as would be the case for linear equations in distributional sense by approximation). Nor were we able to find the literature on $C^{1,\alpha}$-regularity for viscosity inequalities. So we thought this result might be worth recording.
\end{abstract}

\sloppy

\maketitle
\tableofcontents
\sloppy
\section{Introduction}
It is a classical result in the regularity theory of viscosity solutions that viscosity solutions $u: \Omega \subset \R^n \to \R$ to a large class of fully nonlinear elliptic equation
\begin{equation}\label{eq:FD2uf}
 F(D^2 u) = f \quad \mbox{in $\Omega$}
\end{equation}
actually have H\"older continuous gradient, see e.g. \cite[Theorem 8.3]{CC95}. See Section~\ref{s:ingredients} for the precise definition of $F$ we consider here. 

Let us recall that a viscosity solution to \eqref{eq:FD2uf} is a map $u \in C^0(\Omega)$ such that 
\[
 F(D^2 u) \leq f,\mbox{ and  }F(D^2 u) \geq f
\]
both hold in viscosity sense. 
And $F(D^2 u) \leq f$ holds in viscosity sense if for any $\varphi \in C^2(\R^n)$ such that $\varphi - u$ attains its maximum in some $x_0 \in \Omega$ we have
\[
 F(D^2\varphi(x_0)) \leq f(x_0).
\]
Similarly, $F(D^2 u) \geq f$ holds in viscosity sense if for any $\varphi \in C^2(\R^n)$ such that $\varphi - u$ attains its minimum in some $x_0 \in \Omega$ we have
\[
 F(D^2\varphi(x_0)) \geq f(x_0).
\]
For an introduction to the theory of viscosity solutions we refer e.g. to \cite{CC95,K12,K18}.

In this small note we want to record that the $C^{1,\alpha}$-regularity theory for \emph{equations} $F(D^2 u) = f$ also holds for \emph{differential inequalities}. More precisely we have
\begin{theorem}\label{th:main}
Assume that $u \in C^0(\Omega)$ for some $\beta > 0$ solves in viscosity sense
\begin{equation}\label{eq:upde}
 - \Lambda \leq F(D^2 u) \leq \Lambda \quad \mbox{in $\Omega$},
\end{equation}
where $F$ is a uniformly elliptic operator and $1$-homogeneous (see Section~\ref{s:ingredients}), and $\Lambda < \infty$ is a constant. Then $u \in C^{1,\alpha}(\Omega)$ for some $\alpha < 1$.
\end{theorem}

Let us remark that Theorem~\ref{th:main} does not seem to follow (even in the linear case $F(D^2 u) = \lap u$ and even with right-hand side in $f \in L^\infty$) only from considering incremental quotients and using Harnack inequality (as in \cite[\textsection 5.3]{CC95} where the right-hand side is zero). The incremental quotient of $f$ is not uniformly bounded and blows up as $h \to 0$.

The problem that lead us to searching in the literature for Theorem~\ref{th:main} is the following: in \cite{KS18} Khomrutai and the author study a geometric obstacle problem. In this geometric problem one is lead to consider obstacle problems for obstacles $\psi \in C^2$ where the energies is of the form
\[
\int |\nabla u|^2 + u^2 g \quad \mbox{where $u \geq \psi$}.
\]
For $g \geq 0$ and $g \in L^1$ one can show boundedness of $u$. If one has $g$ bounded one obtains H\"older continuity of $u$. In particular, in the latter case one obtains in viscosity sense the following three inequalities.
\[
\begin{split}
\lap u \leq u g \quad & \mbox{in $\Omega$}\\
\lap u = u g \quad & \mbox{in $\{u > \psi\}$}\\
\lap u \geq \lap \psi  \quad &\mbox{in  $\{u = \psi\}$}.
\end{split}
\]
That is, one can find $\Lambda$ such that
\[
 \lap u \leq \Lambda,
\]
and
\[
 \lap u \geq \Lambda,
\]
both hold in viscosity sense, but it is not obvious how to find a priori a function $f$ such that $\lap u = f \in L^\infty$. 
If these inequalities were to hold for distributional solutions one easily gets $C^{1,\alpha}$-regularity, cf. Theorem~\ref{th:weak}. For this linear problem one might hope to use an argument as in \cite{JLM01} for the $p$-Laplacian to show that the inequality is actually true also in a weak sense. 

Another appraoch to prove Theorem~\ref{th:main} might be to appeal to the relation between Viscosity solutions and pointwise strong solutions as in \cite{CCKS96}, and show this to hold for inequalities.

Our choice of proof for Theorem~\ref{th:main} is very similar to the usual arguments used for equations $F(D^2 u) = f \in L^\infty$, namely one uses a blow-up procedure to reduce the regularity theory to the homogeneous solutions. We saw similar arguments appear e.g. in \cite{APR17,LL17, L14, BD15}.

However, while H\"older continuity for solutions of viscosity inequalities are well-established and easily citable, e.g. in \cite{CC95}, we were not able to find in the literature a statement regarding H\"older continuity for the gradient of solutions to such inequalities. The author would have appreciated such a statement recorded somewhere, and thought it might be useful also for others.

Let us also remark that in the weak sense a theorem similar to Theorem~\ref{th:main} holds true -- simply by approximation.
\begin{theorem}\label{th:weak}
Let $A \in \R^{n \times n}$ be a symmetric positive definite matrix, and let $u \in W^{1,2}(\Omega)$, $\Omega \subset \R^n$ open, solve
 \[
  f_1 \leq \div(A \nabla u) \leq f_2 \quad \mbox{in $\Omega$}
 \]
 that is we have for any $\varphi \in C_c^\infty(\Omega)$, $\varphi \geq 0$,
\[
 -\int \langle A \nabla u, \nabla \varphi \rangle \leq \int f_2 \varphi ,
\]
and
\[
 -\int \langle A \nabla u, \nabla \varphi \rangle \geq \int f_1 \varphi .
\]
Then for every Ball $B(2r) \subset \Omega$,
\[
 \|\nabla^2 u\|_{L^p(B(r))} \aleq \|f_1\|_{L^p(B(2r))} +\|f_2\|_{L^p(B(2r))} + \|u\|_{L^2(B(2r))}.
\]
In particular, by Sobolev embedding, if $p > n$ we obtain $C^{1,\alpha}$-regularity estimates for $u$.
\end{theorem}
\begin{proof}
Let $\eta \in C_c^\infty(B(0,1))$, $\eta \equiv 1$ on $B(0,1/2)$, and $0 \leq \eta \leq 1$ on $B(0,1)$ be the usual mollifying kernel and set $\eta_\eps := \eps^{-n} \eta(\cdot/\eps)$. Denote the convolutions with $\eta_\eps$ by $u_\eps := \eta_\eps \ast u$ and $\varphi_\eps := \eta_\eps \ast \varphi$. Moreover we define
\begin{equation}\label{eq:uepspde}
g_\eps  := \div(A\nabla u_\eps)  \in C^\infty(\Omega_{-\eps}).
\end{equation}
Here \[\Omega_{-\eps} := \{x \in \Omega, \dist(x,\partial \Omega) > \eps\}.\]
We have for any $\varphi \in C_c^\infty(\Omega)$, $\varphi \geq 0$,
\[
 \int g_\eps \varphi = -\int \langle A \nabla u_\eps, \nabla \varphi \rangle = -\int \langle A \nabla u, \nabla \varphi_\eps \rangle \leq \int f_2 \ast \eta_\eps\ \varphi.
\]
and likewise 
\[
 \int g_\eps \varphi \geq \int f_1 \ast \eta_\eps\ \varphi.
\]
With the same argument that one uses to prove the fundamental theorem of calculus, namely letting $\varphi$ approximate the dirac-function, we obtain 
\[
 f_1 \ast \eta_\eps \leq g_\eps \leq f_2 \ast \eta_\eps \quad \mbox{pointwise everywhere in $\Omega_{-\eps}$}.
\]
In particular, for $\eps < r$ and $B(2r) \subset \Omega$ we readily obtain for any $p \in (1,\infty)$
\[
 \|g_\eps\|_{L^p(B(r))} \aleq \|f_1\|_{L^p(B(2r))} + \|f_2\|_{L^p(B(2r))}.
\]
Thus, from standard Calderon-Zygmund elliptic theory for the (constant coefficient-) equation \eqref{eq:uepspde} we find
\[
 \|\nabla^2 u_\eps\|_{L^p(B(r))} \aleq \|f_1\|_{L^p(B(2r)} +\|f_2\|_{L^p(B(2r)} + \|u\|_{L^2(B(2r))}
\]
with constants independent of $\eps$.
Since $u_\eps \xrightarrow{\eps \to 0} u$ in $W^{1,2}_{loc}(\Omega)$ we obtain from the boundedness of the $W^{2,p}$-norm of $u_\eps$ that the weak limit $u \in W^{2,p}_{loc}(\Omega)$. Moreover, from weak convergence we have the estimate
\[
 \|\nabla^2 u \|_{L^p(B(r))}\aleq \liminf_{\eps \to 0} \|\nabla^2 u_\eps\|_{L^p(B(r))} \leq \|f_1\|_{L^p(B(2r)} +\|f_2\|_{L^p(B(2r)} + \|u\|_{L^2(B(2r))}
\]
\end{proof}

\section{Ingredients and definitions}\label{s:ingredients}
Denote by $\mathcal{S}^n \subset \R^{n\times n}$ the symmetric matrices and let $F: \R^{n \times n} \to \R$ be a uniformly elliptic operator, that is we shall assume there exists ellipticity constants $0 < \lambda_1 < \lambda_2 < \infty$ such that
\begin{equation}\label{eq:elF}
 \lambda_1 {\rm tr}(N) \leq F(M+N) - F(M) \leq \lambda_2 {\rm tr}(N) \quad \forall M,N \in \mathcal{S}^n, \quad N \geq 0.
\end{equation}
Moreover, we shall assume that $F$ is $1$-homogeneous, i.e. that $F(\sigma N) = \sigma F(N)$.

For solutions $u$ to the homogeneous equations $F(D^2 u) = 0$ we have by e.g. \cite[Corollary 5.7.]{CC95}
\begin{theorem}[$C^{1,\alpha}$ for homogeneous equation]\label{th:homogeneous}
Assume that $F$ is as above,  $\Omega \subset \R^n$ is open and in viscosity sense $u \in C^0(\Omega)$ solves
\[
F(D^2 u) =0 \quad \mbox{in $\Omega$}.
\]
Then $u \in C^{1,\alpha}(\Omega)$ for some $\alpha < 1$.
\end{theorem}

Theorem~\ref{th:main} is thus a consequence of the following
\begin{theorem}\label{th:main2}
Let $\alpha \in (0,1]$ and assume that $F$ is a homogeneous, uniformly elliptic operator as above such that every viscosity solution $v \in C^0(\Omega)$ of the homogeneous equation
\[
 F(D^2 v) = 0 \quad \mbox{in $\Omega$}
\]
satisfies $v \in C^{1,\alpha}$. 

Assume that $u \in C^0(\Omega)$ solves in viscosity sense \eqref{eq:upde}. Then $u \in C^{1,\beta}(\Omega)$ for any H\"older exponent $\beta \in (0,\alpha)$.
\end{theorem}
H\"older regularity of solutions $u$ of differential inequalities in viscosity sense are standard, they follow from Harnack's inequality. See, e.g., \cite[Proposition 4.10]{CC95}.
\begin{lemma}[Uniform H\"older regularity]\label{la:uniformhoelder}
Let $u$ solve \eqref{eq:upde} for $F$ as above. For some $\gamma \in (0,1)$ we have $C^\gamma$-regularity, namely for any ball $B(2r) \subset \Omega$ we have
\[
 [u]_{C^\gamma(B(r))} \leq C(\Lambda,r,\|u\|_{L^\infty(B(2r))})
\]
\end{lemma}

As a last ingredient we need the (standard) result about limits of uniformly converging viscosity (sub/super)-solutions.
\begin{lemma}\label{la:limit}
Let $\Omega \subset \R^n$ open, $u_k \in C^0(\Omega)$, and $\Lambda_k \in \R$ be a sequence of (viscosity) solutions to
\[
 F(D^2 u_k) \leq \Lambda_k \quad \mbox{in $\Omega$},
\]
or
\[
 F(D^2 u_k) \geq \Lambda_k \quad \mbox{in $\Omega$},
\]
respectively.

Assume that $\Lambda_k \to \Lambda_\infty \in \R$ and $u_k$ converges locally uniformly to $u_\infty$. Then $u_\infty$ is a solution in viscosity sense of
\[
 F(D^2 u_\infty) \leq \Lambda_\infty \quad \mbox{in $\Omega$},
\]
or 
\[
 F(D^2 u_\infty) \geq \Lambda_\infty \quad \mbox{in $\Omega$},
\]
\end{lemma}
\begin{proof}
This is of course well known, but we repeat the argument for the $\leq$-case.

Let $u_k \in C^0(\Omega)$ converge locally uniformly to $u_\infty \in C^0(\Omega)$, and assume that 
\begin{equation}\label{eq:pdeuk}
F(D^2 u_k) \leq \Lambda_k
\end{equation}
in viscosity sense, for some constants $\Lambda_k \xrightarrow{k \to \infty} \Lambda$. We will show that then (also in viscosity sense)
\[
F(D^2 u) \leq \Lambda.
\]
So let $\varphi \in C^2(\Omega)$ be a function testfunction for $u$, i.e. assume that $\varphi \leq u$ and $\varphi(x_0) = u(x_0)$. We need to show that 
\begin{equation}\label{eq:gl}
F(D^2 \varphi(x_0)) \leq \Lambda.
\end{equation}
Set
\[
\tilde{\varphi}(x) := \varphi(x) - |x-x_0|^4.
\]
Now we observe that for any $y$ satisfying
\begin{equation}\label{eq:vpyeps}
\tilde{\varphi}(y) - u_k(y) \geq \tilde{\varphi}(x_0) - u_k(x_0)
\end{equation}
we also have
\[
\tilde{\varphi}(y) - u(y) \geq \tilde{\varphi}(x_0) - u(x_0) - 2\|u-u_k\|_{L^\infty}.
\]
Since $u(y) \geq \varphi(y)$ and $\varphi(x_0) = u(x_0)$ we obtain from the definition of $\tilde{\varphi}$,  
\[
- |y-x_0|^4  \geq \varphi(y) - u(y)- |y-x_0|^4  \geq - 2\|u-u_k\|_{L^\infty},
\]
that is any $y$ satisfying \eqref{eq:vpyeps} also satisfies
\[
|y-x_0|^4  \leq 2\|u-u_k\|_{L^\infty} \xrightarrow{k \to \infty} 0.
\]
In particular we can find a sequence $x_k \xrightarrow{k \to \infty} x_0$ such that
\[
\tilde{\varphi}(x_k) - u_k(x_k) = \max_{x} \brac{\tilde{\varphi}(x) - u_k(x)} \geq \tilde{\varphi}(x_0) - u_k(x_0)
\]
That is, $\tilde{\varphi}(x)$ is a testfunction for $u_k$ at $x_k$, and from \eqref{eq:pdeuk} we  get
\[
F(D^2 \tilde{\varphi}(x_k)) \leq \Lambda_k.
\]
From the ellipticity condition \eqref{eq:elF} we also obtain (see \cite[Lemma 2.2]{CC95}) for $M = D^2 \tilde{\varphi}(x_k)$ and $N = D^2 \varphi(x_0)-D^2 \tilde{\varphi}(x_k)$
\[
F(D^2 \varphi(x_0)) \leq F(D^2 \tilde{\varphi}(x_k)) + C(\Lambda)\, |D^2 \tilde{\varphi}(x_k) - D^2 \varphi(x_0)| \leq \Lambda_k + C(\Lambda)\, |D^2 \tilde{\varphi}(x_k) - D^2 \varphi(x_0)|.
\]
But since $x_k \xrightarrow{k \to \infty} x_0$ we have $D^2 \tilde{\varphi}(x_k) \xrightarrow{k \to \infty} D^2 \tilde{\varphi}(x_0) = D^2 \varphi(x_0)$. Thus, we obtain \eqref{eq:gl}.
\end{proof}

\section{Proof of the main theorem}
The heart of the matter is the following decay estimate for the oscillation, we found this kind of argument in \cite[Lemma 3.4]{BD15}.
\begin{proposition}\label{pr:decay}
Let $F$ be as above, and $\alpha$ as in Theorem~\ref{th:main2}. For any $\beta < \alpha$ and any $\lambda_0 \in (0,1)$ there exists $\eps > 0$ and $\lambda \in (0,\lambda_0)$ such that the following holds.

Let $u \in C^0(B(0,1))$ with $\osc_{B(0,1)} u \leq 1$ and
\[
 -\eps \leq F(D^2 u) \leq \eps \quad \mbox{in $B(0,1)$}
\]
Then there exists $q \in \R^n$ such that
\[
 \osc_{B(\lambda)} (u-\langle q, x\rangle_{\R^n}) < \frac{1}{2} \lambda^{1+\beta}. 
\]
\end{proposition}
\begin{proof}
Assume the claim is false for some fixed $\beta < \alpha$ and $\lambda_0 \in (0,1)$. Then we find for every $k \in \N$ functions $u_k \in C^0(B(0,1))$ with $\osc_{B(0,1)} u_k \leq 1$ solving
\[
 -\frac{1}{k}\leq F(D^2 u_k) \leq \frac{1}{k},
\]
but for every $\lambda \in (0,\lambda_0)$ we have
\[
 \inf_{q^\ast \in \R^n} \osc_{B(\lambda)} (u_k-\langle q^\ast, x\rangle_{\R^n}) \geq \frac{1}{2} \lambda^{1+\beta}. 
\]
Without loss of generality we can assume that $u_k(0) = 0$ (since otherwise $u_k - u_k(0)$ satisfies the same assumptions), and since $\osc_{B(0,1)} u_k \leq 1$ we have $\|u_k\|_{\infty} \leq 1$. By Lemma~\ref{la:uniformhoelder} the $u_k$ are uniformly bounded in $C^\alpha$, for some fixed $\alpha > 0$. By Arzela-Ascoli we thus may assume, up to taking a subsequence, that $u_k \to u_\infty$ locally uniformly in $B(0,1)$.

In view of Lemma~\ref{la:limit} we find that $u_\infty$ solves the homogeneous equation
\[
F(D^2 u_\infty) =0 \quad \mbox{in $B(0,1)$}.
\]
From the assumptions of Theorem~\ref{th:main2} we know that $u_\infty \in C^{1,\alpha}$. From Taylor's theorem we have thus for any $\lambda \in (0,1/4)$,
\[
 \inf_{q^\ast \in \R^n} \osc_{B(\lambda)} (u_\infty-\langle q^\ast, x\rangle_{\R^n}) \aleq [u_\infty]_{C^{1,\alpha}(B(0,1/2))}\, \lambda^{1+\alpha}.
\]
On the other hand, by locally uniform convergence of $u_k$ we have for any $\lambda \in (0,\lambda_0)$.
\[
 \inf_{q^\ast \in \R^n} \osc_{B(\lambda)} (u_\infty-\langle q^\ast, x\rangle_{\R^n}) \geq \frac{1}{2} \lambda^{1+\beta}. 
\]
That is, we have that for all $\lambda \in (0,1/4)$, $\lambda < \lambda_0$.
\[
 \lambda^{\beta-\alpha} \leq [u_\infty]_{C^{1,\alpha}}
\]
Since $\beta < \alpha$ this is impossible for very small $\lambda$.
\end{proof}

Iterating Proposition~\ref{pr:decay} we obtain
\begin{corollary}\label{co:rescaleddecay}
Let $F$ be as above, and $\alpha$ as in Theorem~\ref{th:main2}. For any $\beta < \alpha$ and any $\lambda_0 \in (0,1)$ there exist $\eps > 0$ and $\lambda \in (0,\lambda_0)$ such that the following holds.

Assume $u$ solves
\begin{equation}\label{eq:co1pde}
 -\eps \leq F(D^2 u) \leq \eps \quad \mbox{in $B(0,1)$}
\end{equation}
and
\[
\osc_{B(0,1)} u < 1.
\]
Then for any $k \in \N \cup \{0\}$, there exists $q_k \in \R^n$ such that
\[
\lambda^{-k(1+\beta)}\, \osc_{B(\lambda^k)} \, \brac{u(x) - q_k \cdot x} < 2^{-k}. 
\]
\end{corollary}
\begin{proof}
Let $\lambda_0$ w.l.o.g. be such that $2\lambda_0^{1-\beta} < 1$ and let $\lambda \in (0,\lambda_0)$ be from Proposition~\ref{pr:decay}.
For $k \in \N \cup \{0\}$ we set
\[
u_k(x) := 2^{k} \lambda^{-k(1+\beta)}\, \brac{u(\lambda^k x) - q_k \cdot \lambda^k x},
\]
where $q_0 = 0$ and $q_k \in \R^n$, $k \geq 1$, remains to be chosen.

Regardless of the choice of the constant vector $q_k$ we obtain from \eqref{eq:co1pde}, for every $k \in \N \cup \{0\}$,
\[
-2^k \lambda^{k(1-\beta)} \eps   \leq F(D^2 u_k) \leq 2^k \lambda^{k(1-\beta)} \eps \quad \mbox{in $B(0,1)$}.
\]
By the choice of $\lambda_0$ and since $\lambda \in (0,\lambda_0)$ we have in particular for every $k \in \N \cup \{0\}$,
\begin{equation}\label{eq:ukpde}
-\eps \leq F(D^2 u_k) \leq \eps \quad \mbox{in $B(0,1)$}.
\end{equation}
The claim follows, once we show 
\begin{equation}\label{eq:oscukl1}
\osc_{B(0,1)} u_k < 1 \quad \mbox{for all $k \in \N$}.
\end{equation}
We prove \eqref{eq:oscukl1} by induction, for $k = 0$ this holds already by assumption.
Fix $k \in \N$. As induction hypothesis we assume the following holds
\[
\osc_{B(0,1)} u_{k-1} < 1.
\]
In view of \eqref{eq:ukpde} we can apply Proposition~\ref{pr:decay}, and find $\tilde{q_k} \in \R^n$ such that
\[
 2\lambda^{-1-\beta} \osc_{B(\lambda)} \brac{u_{k-1} - \langle \tilde{q}_k, x\rangle_{\R^n}} <1.
\]
That is
\[
 2\lambda^{-1-\beta} \osc_{B(1)} \brac{u_{k-1}(\lambda \cdot) - \langle \lambda \tilde{q}_k, x\rangle_{\R^n}} <1.
\]
By the definition of $u_{k-1}$,
\[
2^{k} \lambda^{-k(1+\beta)}\osc_{B(1)}  \brac{
u(\lambda^k x) - \left \langle q_{k-1} 
 - 2^{1-k}\lambda^{(k-1)(1+\beta)} \lambda^{1-k} \tilde{q}_k,\ \lambda^k x\right \rangle_{\R^n}} <1.
\]
so if we set
\[
q_k := q_{k-1} 
 - 2^{1-k}\lambda^{(k-1)(1+\beta)} \lambda^{1-k} \tilde{q}_k,
\]
we have obtained
\[
\osc_{B(1)}(u_k) < 1.
\]
That is, by induction, \eqref{eq:oscukl1} holds for any $k \in \N \cup 0$.
\end{proof}

\begin{corollary}\label{co:rescaleddecay2}
Let $F$ be as above, and $\alpha$ as in Theorem~\ref{th:main2}. For any $\beta < \alpha$ let $u$ solve for some ball $B(R) \subset \Omega$
\[
-\Lambda \leq F(D^2 u)\leq \Lambda \quad \mbox{in $B(R)$}
\]
Then
\[
\sup_{r < R} r^{-1-\beta}\inf_{q \in \R^n} \osc_{B(r)} (u-\langle q, x\rangle_{\R^n}) \leq C(\beta,\alpha,\Lambda,R,\osc_{B(R)} u).
\]
\end{corollary}
\begin{proof}
By otherwise considering $u_{\kappa,R} := \kappa^{-1} u(R x)$ for
\[
\kappa := \frac{\Lambda}{\eps} + R^2 + \osc_{B(R)} u + 1,
\]
we can assume that $R = 1$, $\Lambda < \eps$ and $\osc_{B(1)} u < 1$. Here $\eps$ is from Corollary~\ref{co:rescaleddecay}.

Denoting for the ball $B(r)$
\[
 \Phi(B(r)) := r^{-1-\beta}\inf_{q \in \R^n} \osc_{B(r)} (u_\infty-\langle q, x\rangle_{\R^n}),
\]
we get from Corollary~\ref{co:rescaleddecay} for any $r \in (\lambda^{k-1},\lambda^{k})$
\[
 \Phi(B(r)) \leq \lambda^{-1-\beta} \Phi(B(\lambda^k)) \leq C(\lambda) 2^{-k} \Phi(B(1)) \leq C(\lambda) r^{\frac{\log 2}{-\log \lambda}}\Phi(B(1)).
\]
This implies for $\sigma := \frac{\log 2}{-\log \lambda} > 0$ 
\[
 \sup_{r < R} r^{-\sigma} \Phi(B(r)) \leq C(\lambda)\, \Phi(B(1)).
\]
Dropping the $\sigma$, the claim is now proven.
\end{proof}

\begin{proof}[Proof of Theorem~\ref{th:main2}]
Let $K \subset \Omega$ be a compact set. By a covering argument for any $\beta < \alpha$ we obtain from Corollary~\ref{co:rescaleddecay2}
\[
\sup_{x \in K, r < \dist(x,\partial \Omega)} r^{-1-\beta}\inf_{q \in \R^n} \osc_{B(r)} (u-\langle q, x\rangle_{\R^n}) < \infty
\]
This readily implies that $u \in C^{1,\beta}(K)$ for any $\beta < \alpha$, see, e.g. \cite{C64}. See also \cite[Theorem 4.4.]{RSS13}.
\end{proof}

\subsection*{Acknowledgment}
Partial support by the Daimler and Benz foundation through grant no. 32-11/16 and Simons foundation through grant no 579261 is gratefully acknowledged.

The author would like to thank Quoc-Hung Nguyen, Cyril Imbert, Erik Lindgren, Qing Liu, Russel Schwab, and Pablo Stinga for helpful suggestions.

\bibliographystyle{abbrv}%
\bibliography{bib}%

\end{document}